\documentclass[12pt]{amsart}
\usepackage[margin=1.1in]{geometry}
\usepackage{pb-diagram}

\usepackage{amssymb,amsthm,amsmath,latexsym,stmaryrd}
\usepackage{graphicx}
\usepackage[all]{xy}

\newcommand{\complex}{{\mathbb C}}

\newcommand{\ball}{{\mathbb B}}

\theoremstyle{plain}
        \newtheorem{theorem}{Theorem}[section]
        \newtheorem{lemma}[theorem]{Lemma}
        \newtheorem{remark}[theorem]{Remark}
        \newtheorem{proposition}[theorem]{Proposition}
        \newtheorem{corollary}[theorem]{Corollary}

\theoremstyle{definition}
        \newtheorem{definition}[theorem]{Definition}

\title[Index Theorem and Resolution]{A New Index Theorem for Monomial Ideals by Resolutions}
\author{Ronald G. Douglas}
\address{Department of Mathematics, Texas A$\&$M University, College Station, TX, 77843}
\email{rdouglas@math.tamu.edu}

\author{Mohammad Jabbari}
\address{Department of Mathematics, Washington University, St. Louis, Missouri, USA, 63130}
\email{jabbari@wustl.edu}

\author{Xiang Tang}
\address{Department of Mathematics, Washington University, St. Louis, Missouri, USA, 63130}
\email{xtang@math.wustl.edu}

\author{Guoliang Yu}
\address{Department of Mathematics, Texas  A$\&$M University, College Station, TX, 77843, Shanghai Center for Mathematical Sciences, Fudan University, Shanghai, China, 200433}
\email{guoliangyu@math.tamu.edu}

\begin{document}
\begin{abstract}
We prove an index theorem for the quotient module of a monomial ideal.  We obtain this result by resolving the monomial ideal by a sequence of Bergman space like essentially normal Hilbert modules. 
\end{abstract}
\maketitle
\section{Introduction}
Let $\mathbb{B}^m$ be the unit ball in the complex space $\mathbb{C}^m$, and $L^2_a(\mathbb{B}^m)$ be the Bergman space of square integrable holomorphic functions on $\mathbb{B}^m$. Denote $A$ to be the algebra $\mathbb{C}[z_1, \cdots, z_m]$ of polynomials of $m$ generators.  The algebra $A$ plays two roles in our study. One is that $A$ is a dense subspace of the Hilbert space $L^2_a(\mathbb{B}^m)$, the other is that $A$ acts on $L^2_a(\mathbb{B}^m)$ by Toeplitz operators. 

In this article we are interested in an ideal $I$ of $A$ generated by monomials. Let $\overline{I}$ be the closure of $I$ in $L^2_a(\mathbb{B}^m)$, and $Q_I$ be the quotient Hilbert space $L^2_a(\ball^m)/\overline{I}$. The first author proved in \cite[Theorem 2.1]{do:index} that the the Toeplitz operators $T_{z_i}$ on $\overline{I}$, $i=1,\cdots, m$, and the quotient $Q_I$ are essentially normal\footnote{Arveson \cite[Corollary 2.2]{ar:monomial} proved the similar result on the Drury-Arveson space.}, i.e. the following commutators are compact,
\[
[T_{z_i}|_{\overline{I}}, \big(T_{z_j}|_{\overline{I}}\big)^*]\in \mathcal{K}(\overline{I}),\ \text{and}\ [T_{z_i}|_{Q_I}, \big(T_{z_j}|_{Q_I}\big)^*]\in \mathcal{K}(Q_I),\ i,j=1,\cdots, m. 
\]
Let $\mathfrak{T}(Q_I)$ be the unital $C^*$-algebra generated by the Toeplitz operators $T_{z_i}|_{Q_I}$, $i=1, \cdots, m$. The above essentially normal property of the Toeplitz operators gives the following extension sequence
\[
0\longrightarrow \mathcal{K}\longrightarrow \mathfrak{T}(Q_I)\longrightarrow C(\sigma^e_I)\longrightarrow 0,
\]
where $\sigma^e_I$ is the essential spectrum space of the algebra $\mathfrak{T}(Q_I)$ and $\mathcal{K}$ is the algebra of compact operators. The index problem we want to  answer in this article is to provide a good description of the above $K$-homology class. 

The main difficulty in answering the above question is that the ideal $I$ in general fails to be radical. This makes the geometric ideas introduced in \cite{do-ta-yu:grr} and \cite{ee:geometric} impossible to apply directly. The seed of the main idea in this article is the following observation discussed in \cite[Section 5.2]{do-ta-yu:grr}. For $m=2$, consider the ideal $I=\langle z_1^2\rangle\triangleleft A=\complex[z_1, z_2]$. The quotient $Q_I$ can be written as the sum of two space
\[
L^2_{a,1}(\mathbb{D})\oplus L^2_{a,2}(\mathbb{D}),
\]
where $\mathbb{D}$ is the unit disk inside the complex plane $\complex$, and $L^2_{a,1}(-)$ (and $L^2_{a,2}(-)$) is the weighted Bergman space with respect to the weight defined by the defining function $1-|z|^2$ (and $(1-|z|^2)^2$).  Define the restriction map  $R_I: L^2_a(\ball^2)\to L^2_{a,1}(\mathbb{D})\oplus L^2_{a, 2}(\mathbb{D})$ by 
\[
R_I(f):=(f|_{z_1=0}, \frac{\partial f}{\partial z_1}|_{z_1=0}).
\]
Then it is not hard to introduce a Hilbert $A=\complex[z_1, z_2]$-module structure on $L^2_{a,1}(\mathbb{D})\oplus L^2_{a,2}(\mathbb{D})$ so that the following exact sequence of Hilbert modules holds, 
\[
0\longrightarrow \overline{I}\longrightarrow L^2_a(\ball^2)\longrightarrow L^2_{a,1}(\mathbb{D})\oplus L^2_{a, 2}(\mathbb{D})\longrightarrow 0.
\]
It is well known that the Toeplitz operators on the (weighted) Bergman space  $L^2_{a,1}(\mathbb{D})$ and $ L^2_{a, 2}(\mathbb{D})$ are essentially normal. In \cite[Section 5.2]{do-ta-yu:grr}, we observe that one can conclude the essential normal property of the Hilbert module $\overline{I}$ and $Q_I$ from the above exact sequence. Furthermore, it is not hard to prove that the quotient Hilbert module $Q_I$ is isomorphic to the module $ L^2_{a,1}(\mathbb{D})\oplus L^2_{a, 2}(\mathbb{D})$. And we can identify the extension class $[\mathfrak{T}(Q_I)]$ associated to $Q_I$ with the one associated to $ L^2_{a,1}(\mathbb{D})\oplus L^2_{a, 2}(\mathbb{D})$. 

In this article, we extend the above example to an arbitrary ideal $I$ of $L^2_a(\ball^m)$ generated by monomials. By considering an ideal generated by two monomials, we realize that it is more natural to work with long exact sequences of Hilbert modules, instead of short ones. The following is our main theorem. 

\begin{theorem}\label{thm:main-intro}
Let $I$ be an ideal  of $\complex[z_1,\cdots, z_m]$ generated by monomials, and $\overline{I}$ be its closure in the Bergman space $L^2_a(\ball^m)$. There are Bergman space like Hilbert $A$-modules $\mathcal{A}_0=L^2_a(\ball^m), \mathcal{A}_1, \cdots, \mathcal{A}_k$ together with bounded $A$-module morphisms $\Psi_i: \mathcal{A}_i\to \mathcal{A}_{i+1}$, $i=0, \cdots, k-1$ such that the following exact sequence of Hilbert modules holds
\[
0\longrightarrow \overline{I}\hookrightarrow L^2_a(\ball^m)\stackrel{\Psi_0}{\longrightarrow} \mathcal{A}_1\stackrel{\Psi_1}{\longrightarrow}\cdots \stackrel{\Psi_{k-1}}{\longrightarrow }\mathcal{A}_k\longrightarrow 0.
\]
\end{theorem}

As a corollary of Theorem \ref{thm:main-intro}, we obtain a new proof of essentially normal property of the ideal $\overline{I}$ and its quotient $Q_I$. Moreover, Theorem \ref{thm:main-intro} allows us to identify the extension class associated to $\mathfrak{T}(Q_I)$ geometrically. We compute it in the following theorem. 

\begin{theorem} \label{thm:k-hom-intro} (Theorem \ref{thm:k-hom})
Let $\mathfrak{T}(\mathcal{A}_i)$ be the unital $C^*$-algebra generated by Toeplitz operators on $\mathcal{A}_i$, and $\sigma^e_{i}$ be the associated essential spectrum space.  In $K_1(\sigma^e_1\cup \cdots \cup \sigma^e_k)$, the following equation holds, 
\[
[\mathfrak{T}(Q_I)]=[\mathfrak{T}(\mathcal{A}_1)]-[\mathfrak{T}(\mathcal{A}_2)]+\cdots +(-1)^{k-1} [\mathfrak{T}(\mathcal{A}_k)],
\]
\end{theorem}

As is explained in Section \ref{subsec:geometry} and Remark \ref{rmk:geometry}, every algebra $\mathfrak{T}(\mathcal{A}_i),\ i=1,\cdots, k$, can be identified as the algebra of Toeplitz operators on square integrable holomorphic sections of a hermitian vector bundle on  a disjoint union of subsets of $\ball^m$. This geometric interpretation allows us to use the ideas developed in  \cite{do-ta-yu:grr} and \cite{ee:geometric} to study the ``geometry" of the algebra $\mathfrak{T}(Q_I)$. 

We would like to remark that although our Theorem \ref{thm:main-intro} is stated for the closure of a monomial ideal inside the Bergman space $L^2_a(\ball^m)$, the same proof also works for the closure of any monomial ideal inside more general spaces, e.g. weighted Bergman spaces $L^2_{a, s}(\ball^m)$ and the Drury-Arveson space. 

Our result in Theorem \ref{thm:main-intro} can be viewed as a ``resolution" (\cite{hi:resolution}) of the ideal $\overline{I}$ by essentially normal Hilbert modules $\mathcal{A}_0, \cdots, \mathcal{A}_k$. Such an idea of ``resolution" goes back the first author's work in \cite{do:index}, and the study in this article should not be limited to monomial ideals.  In Section \ref{subsec:none-monomial}, we explain that a similar construction also works on a more general ideal in $\complex[z_1, z_2, z_3]$. We hope to report in the near future about a systematic study on extending ideas from Theorem \ref{thm:main-intro} to more  general ideals $I$. 

This article is organized as follows. In Section \ref{sec:generalized-Bergman}, we will introduce the main building block in our construction. In Section \ref{sec:resolution}, we will construct the Hilbert modules $\mathcal{A}_i$ and morphisms $\Psi_i,\ i=0, \cdots, k$ in Theorem \ref{thm:main-intro}, and prove  the main Theorem. Our proof is inspired by the corresponding algebraic study on monomial ideals \cite{co-li-osh:ideals}, \cite{he-hi:monomial}. We will end our paper by exhibiting our constructions on concrete examples in Section \ref{sec:example}. In particular, a non-monomial ideal example is discussed in Section \ref{subsec:none-monomial}. \\

\noindent{\bf Acknowledgements:} All authors are partially supported by U.S. NSF. Yu is partially supported by NSFC11420101001. 

\section{Generalized Bergman Space and the Associated Toeplitz Operators}\label{sec:generalized-Bergman}

In this section, we introduce and study a building block in the construction of the resolution in Theorem \ref{thm:main-intro}. 

\subsection{Notations}\label{subsec:box-hilbert}

We start with fixing some notations. For a positive integer $q$, we use the symbol $S_q(m)$ to denote the set of $q$-shuffles of the set $[m]=\{1,\cdots, m\}$, i.e. 
\[
S_{q}(m):=\{\mathfrak{j}:=(j^1, \cdots, j^q)| 1\leq j^1< j^2< \cdots < j^q\leq m\}. 
\]
Let $\mathbb{N}$ be the set of all nonnegative integers. For any $\mathfrak{i}=(i^1, \cdots, i^q)\in \mathbb{N}^q$, we use  $|\mathfrak{i}|$ to denote the sum $i^1+\cdots+i^q$.  

Fix $\mathfrak{j}=(j^1, \cdots, j^q)\in S_q(m)$, and $\mathfrak{b}=(b^1, \cdots, b^q)\in \mathbb{N}^q$.  We associate a subset $\mathcal{B}^{\mathfrak{b}}_{\mathfrak{j}}\subseteq \mathbb{N}^m$  defined by 
\[
\mathcal{B}^{\mathfrak{b}}_{\mathfrak{j}}:=\{(n^1,\cdots, n^m)\in \mathbb{N}^m| n^i\in \mathbb{N},\ i\notin \mathfrak{j},\ 0\leq n^i \leq b^{k},\ i=j^k\in \mathfrak{j}\}. 
\]
We call $\mathcal{B}^{\mathfrak{b}}_{\mathfrak{j}}$ the box associated to $\mathfrak{j}$ and $\mathfrak{b}$.

In the following, we introduce a Hilbert space $\mathcal{H}^{\mathfrak{b}}_{\mathfrak{j}}$ as a closed subspace of $L^2_a(\ball^m)$. 

On $\ball^m$, consider the weighted Bergman space $L^2_{a, s}(\ball^m)$ defined by square integrable holomorphic functions on $\ball^m$ with respect to the norm 
\[
||f||^2_{L^2_{a,s}}=\int_{\ball^m} |f(z)|^2 (1-|z|^2)^s \frac{(m+s)!}{m! s!} {\rm d}V(z), 
\]
where ${\rm{d}}V$ is the normalized Lebesgue measure on $\ball^m$. This Hilbert space has the following standard orthonormal basis
\begin{equation}\label{eq:basis}
\left\{z^{\mathfrak{n}}:=\frac{z_1^{n_1}\cdots z_m^{n_m}}{\sqrt{\omega_s(\mathfrak{n})}}\bigg|\ \mathfrak{n}=(n_1, \cdots, n_m)\in \mathbb{N}^m\right\}
\end{equation}
where $\omega_s(\mathfrak{n}):=\frac{n_1!\cdots n_m!(m+s)!}{(n_1+\cdots +n_m+s+m)!} $. 

\begin{definition}\label{defn:hilbert}The Hilbert space $\mathcal{H}^{\mathfrak{b}}_{\mathfrak{j}}$ is a closed subspace of $L^2_a(\ball^m)$ consisting of functions $f\in L^2_a(\ball^m)$ whose expansion $\sum_{\mathfrak{n}} f_{\mathfrak{n}}z^\mathfrak{n}$ with respect to the orthonormal basis $z^\mathfrak{n}$ satisfies 
\[
f_{\mathfrak{n}}=0,\ \text{for}\ \mathfrak{n}\notin \mathcal{B}^{\mathfrak{b}}_{\mathfrak{j}}. 
\]
\end{definition}

We have the following orthonormal basis for the Hilbert space $\mathcal{H}_{\mathfrak{j}}^{\mathfrak{b}}$,
\[
z^{\mathfrak{n}}:=\frac{z_1^{n^1}...z_{m}^{n^m}}{\sqrt{\omega_0(\mathfrak{n})}},
\]
where $0\leq n_{j^1}\leq b^1, \cdots, 0\leq n_{j^q}\leq b^q$, and the index $n_i$ for $i\notin \mathfrak{j}$ belongs to $\mathbb{N}$.  In terms of this basis, an element $X\in \mathcal{H}_{\mathfrak{j}}^{\mathfrak{b}}$ can be written as
\begin{equation}\label{eq:expression}
X=\sum_{0\leq n^{j^1}\leq b^1, \cdots , 0\leq n^{j^q}\leq b^q} X_{n^1\cdots n^m} z^{\mathfrak{n}}.
\end{equation}

In the following, we define a representation of the polynomial algebra $A=\complex[z_1, \cdots, z_m]$ on the Hilbert space $\mathcal{H}^{\mathfrak{b}}_{\mathfrak{j}}$. For $p=1,\cdots, m$, define the operator $T^{\mathfrak{j}, \mathfrak{b}}_{z_p} $ on the Hilbert space $\mathcal{H}^{\mathfrak{b}}_{\mathfrak{j}}$ by 
\[
T^{\mathfrak{j}, \mathfrak{b}}_{z_p}(z^{\mathfrak{n}}):=\left\{\begin{array}{ll} z_p z^{\mathfrak{n}},&\text{if}\ (n^1,\cdots, n^p+1, \cdots, n^m)\in \mathcal{B}^{\mathfrak{b}}_{\mathfrak{j}},\\ 0,&
\text{otherwise}. \end{array}\right.
\] 

\begin{lemma}\label{lem:bdd-rep} The above operators $\{T^{\mathfrak{j}, \mathfrak{b}}_{z_i}\}_{i=1}^m$ define a bounded representation of the algebra $A$ on the space $\mathcal{H}_{\mathfrak{j}}^{\mathfrak{b}}$. 
\end{lemma}
\begin{proof}Let $T_{z_p}$ be the standard Toeplitz representation of $z_p$ on the Bergman space $L^2_a(\ball^m)$, and $P^{\mathfrak{b}}_{\mathfrak{j}}$ be the orthogonal projection from $L^2_a(\ball^m)$ onto the closed subspace $\mathcal{H}_{\mathfrak{j}}^{\mathfrak{b}}$. Then the operator $T^{\mathfrak{j}, \mathfrak{b}}_{z_p}$ can be identified with $P^{\mathfrak{b}}_{\mathfrak{j}}T_{z_p}P^{\mathfrak{b}}_{\mathfrak{j}}$. It follows that $P^{\mathfrak{b}}_{\mathfrak{j}}T_{z_p}P^{\mathfrak{b}}_{\mathfrak{j}}$ is bounded. 

For $p\ne p'$, consider $\mathfrak{n}=(n^1, \cdots, n^m)\in \mathcal{B}^{\mathfrak{b}}_{\mathfrak{j}}$, and the corresponding basis vector $z^{\mathfrak{n}}$.

If $(n^1, \cdots, n^p+1, \cdots, n^{p'}+1, \cdots, n^m)\in \mathcal{B}^{\mathfrak{b}}_{\mathfrak{j}}$, both $(n^1, \cdots, n^p+1, \cdots, n^{p'}, \cdots, n^m)$ and $(n^1, \cdots, n^p,\cdots, n^{p'}+1, \cdots, n^m)$ belong to $\mathcal{B}^{\mathfrak{b}}_{\mathfrak{j}}$. Hence, we have 
\[
T^{\mathfrak{j}, \mathfrak{b}}_{z_{p'}}\big(T^{\mathfrak{j}, \mathfrak{b}}_{z_p}(z^\mathfrak{n})\big)=T^{\mathfrak{j}, \mathfrak{b}}_{z_p}\big(T^{\mathfrak{j}, \mathfrak{b}}_{z_{p'}}(z^{\mathfrak{n}})\big)=z_pz_{p'}z^{\mathfrak{n}}. 
\]

If $(n^1, \cdots, n^p+1, \cdots, n^{p'}+1, \cdots, n^m)\notin \mathcal{B}^{\mathfrak{b}}_{\mathfrak{j}}$ but $(n^1, \cdots, n^p, \cdots, n^{p'}, \cdots, n^m)\in \mathcal{B}^{\mathfrak{b}}_{\mathfrak{j}}$, either $p=j^s$ and $n^p=b^s$ for some $s$ or $p'=j^s$ and $n^{p'}=b^s$. Without loss of generality, we assume that $p=j^s$ and $n^p=b^s$ for some $s$. Then  $(n^1, \cdots, n^p+1, \cdots, n^{p'}, \cdots, n^m)$ does not belong to the box $\mathcal{B}^{\mathfrak{b}}_{\mathfrak{j}}$.  We conclude that 
\[
T^{\mathfrak{j}, \mathfrak{b}}_{z_p}(z^\mathfrak{n})=0,\ \text{and}\ T^{\mathfrak{j}, \mathfrak{b}}_{z_{p'}}\big(T^{\mathfrak{j}, \mathfrak{b}}_{z_p}(z^\mathfrak{n})\big)=T^{\mathfrak{j}, \mathfrak{b}}_{z_p}\big(T^{\mathfrak{j}, \mathfrak{b}}_{z_{p'}}(z^{\mathfrak{n}})\big)=0. 
\]

Summarizing the above discussion, we conclude that 
$$T^{\mathfrak{j}, \mathfrak{b}}_{z_{p'}}T^{\mathfrak{j}, \mathfrak{b}}_{z_p}=T^{\mathfrak{j}, \mathfrak{b}}_{z_p}T^{\mathfrak{j}, \mathfrak{b}}_{z_{p'}}.$$
\end{proof}

\subsection{Essentially normal property}
In this subsection, we prove the following property of the Hilbert $A$-module $\mathcal{H}_{\mathfrak{j}}^{\mathfrak{b}}$. 
\begin{proposition}\label{prop:ess-norm-box} The following commutators are compact. 
\[
[{T^{\mathfrak{j}, \mathfrak{b}}_{z_s}}^*, T^{\mathfrak{j}, \mathfrak{b}}_{z_t}] \in \mathcal{K}(\mathcal{H}_{\mathfrak{j}}^{ \mathfrak{b}}),\ \forall s,t=1, \cdots, m. 
\]
Therefore, $\mathcal{H}_{\mathfrak{j}}^{ \mathfrak{b}}$ is an essentially normal Hilbert $A$-module. 
\end{proposition}
Consider the monomial ideal $I=\langle z^{b^1+1}_{j^1}, \cdots, z^{b^q+1}_{j^q}\rangle$ of the polynomial algebra $A$. Then the Hilbert module $\mathcal{H}_{\mathfrak{j}}^\mathfrak{b}$ can be identified with the quotient module $Q_I:=L^2 _a(\ball^m)/\overline{I}$. And Proposition \ref{prop:ess-norm-box} follows from  \cite{ar:monomial} and \cite{do:module}. Here, we outline its proof here for completeness. 

\begin{proof}
Let $P^{\mathfrak{b}}_{\mathfrak{j}}$ be the orthogonal projection from $L^2_a(\ball^m)$ onto the closed subspace $\mathcal{H}_{\mathfrak{j}}^{\mathfrak{b}}$. As $T^{\mathfrak{j}, \mathfrak{b}}_{z_s}$ can be written as $P^{\mathfrak{b}}_{\mathfrak{j}} T_{z_s} P^{\mathfrak{b}}_{\mathfrak{j}}$, it suffices to prove the commutator 
$[P^{\mathfrak{b}}_{\mathfrak{j}}, T_{z_s}]$ is compact. 

For $\mathfrak{n}\in  \mathcal{B}^{\mathfrak{b}}_{\mathfrak{j}}$, $P^{\mathfrak{b}}_{\mathfrak{j}}T_{z_s}(z^{\mathfrak{n}})$ is computed as follows, 
\[
P^{\mathfrak{b}}_{\mathfrak{j}}T_{z_s}(z^\mathfrak{n})=\left\{ \begin{array}{ll} \sqrt{\frac{\omega_0(n^1\cdots (n^s+1)\cdots n^m)}{\omega_0(n^1\cdots n^m)}} z^{n^1\cdots (n^s+1)\cdots n^m},&\text{if}\ (n^1\cdots (n^s+1)\cdots n^m)\in  \mathcal{B}^{\mathfrak{b}}_{\mathfrak{j}}\\ 0,&\text{otherwise}.
\end{array}\right.
\]
Similarly, $T_{z_s}P^{\mathfrak{b}}_{\mathfrak{j}}(z^{\mathfrak{n}})$ is computed as follows,
\[
T_{z_s}P^{\mathfrak{b}}_{\mathfrak{j}}(z^{\mathfrak{n}})=\left\{ \begin{array}{ll} \sqrt{\frac{\omega_0(n^1\cdots (n^s+1)\cdots n^m)}{\omega_0(n^1\cdots n^m)}} z^{n^1\cdots (n^s+1)\cdots n^m},&\text{if}\ (n^1\cdots n^s\cdots n^m)\in  \mathcal{B}^{\mathfrak{b}}_{\mathfrak{j}}\\ 0,&\text{otherwise}.
\end{array}\right.
\]

Hence, $[P^{\mathfrak{b}}_{\mathfrak{j}}, T_{z_s}](z^{\mathfrak{n}})$ is computed as follows,
\[
[P^{\mathfrak{b}}_{\mathfrak{j}}, T_{z_s}](z^{\mathfrak{n}}):=\left\{\begin{array}{ll}
-\sqrt{\frac{\omega_0(n^1\cdots (b^{k}+1)\cdots n^m)}{\omega_0(n^1\cdots b^k\cdots n^m)}} z^{n^1\cdots (n^s+1)\cdots n^m},&\text{if}\ (n^1\cdots n^s\cdots n^m)\in  \mathcal{B}^{\mathfrak{b}}_{\mathfrak{j}},\\
& \text{and}\ s=j^k,\ \text{and}\ n^s=b^k\ \text{for some}\ k,\\ 0,&\text{otherwise}.
\end{array}\right.
\]
We observe that the weight $\frac{\omega_0(n^1\cdots (b^{k}+1)\cdots n^m)}{\omega_0(n^1\cdots b^k\cdots n^m)}$ converges to zero as $||(n^1, \cdots, b^k, \cdots, n^m)||\to \infty$. From this, we can conclude that the commutator $[P^{\mathfrak{b}}_{\mathfrak{j}}, T_{z_s}]$ is compact. 
\end{proof}

\begin{remark}It is not hard to check in the above proof of Proposition \ref{prop:ess-norm-box} that the commutator $[P^{\mathfrak{b}}_{\mathfrak{j}}, T_{z_s}]$ and therefore also $[{T^{\mathfrak{j}, \mathfrak{b}}_{z_s}}^*, T^{\mathfrak{j}, \mathfrak{b}}_{z_t}] $ belong to the Schatten-$p$ ideal for $p>m-q$. 
\end{remark}

\subsection{Geometry of the Hilbert space $\mathcal{H}_{\mathfrak{j}}^{\mathfrak{b}}$}\label{subsec:geometry}

In this subsection, we discuss briefly the geometry of the Hilbert space $\mathcal{H}_{\mathfrak{j}}^{\mathfrak{b}}$ introduced in Section \ref{subsec:box-hilbert}. 

Let $\ball_{\mathfrak{j}}$ be the subset of $\ball^m$ cut by the hyperplanes $H_{j^i}:=\{z_{j^i}=0\}$, $i=1, \cdots, q$, i.e. 
\[
\ball_{\mathfrak{j}}:=\{(z_1, \cdots, z_m)\in \ball^m|\ z_{j^1}=\cdots =z_{j^q}=0\}. 
\]
Observe that $\ball_{\mathfrak{j}}$ is the unit ball inside the subspace 
\[
H_{\mathfrak{j}}:=\{(z_1,\cdots, z_m)| z_{j^1}=\cdots=z_{j^q}=0\},
\]
which is isomorphic to the standard complex space $\mathbb{C}^{m-q}$. 

For $\mathfrak{i}=(i^1, \cdots ,i^q)\in \mathbb{N}^q$, we consider the following weighted Bergman space $L^2_{a, q+|\mathfrak{i}|}(\ball_{\mathfrak{j}})$. And given any $\mathfrak{b}\in \mathbb{N}^q$, we consider the following Hilbert space 
\[
\widetilde{\mathcal{H}}_{\mathfrak{j}}^{\mathfrak{b}}:=\bigoplus _{\mathfrak{i}\in \mathbb{N}^q, i^1\leq b^1, \cdots, i^q\leq b^q} L^2_{a, q+|\mathfrak{i}|}(\ball_{\mathfrak{j}}). 
\]

We define a map $R^{\mathfrak{b}}_{\mathfrak{j}}: \mathcal{H}_{\mathfrak{j}}^{\mathfrak{b}}\to \widetilde{\mathcal{H}}_{\mathfrak{j}}^{\mathfrak{b}}$ by 
\[
R^{\mathfrak{b}}_{\mathfrak{j}}(f):=\sum_{i^1\leq b^1, \cdots, i^q\leq b^q} R^{\mathfrak{i}}_{\mathfrak{j}}(f),
\]
where $R^{\mathfrak{i}}_{\mathfrak{j}}(f)\in L^2_{a, q+|\mathfrak{i}|}(\ball_{\mathfrak{j}})$ is defined as
\[
R^{\mathfrak{i}}_{\mathfrak{j}}(f):=\frac{\partial^{ i^1+\cdots +i^q} f}{\partial z_{j_1}^{i^1}\cdots \partial z_{j_q}^{i^q}}\Big|_{\ball_{\mathfrak{j}}}.
\]

A straightforward computation on the orthonormal basis gives the following property, and we leave the detail proof to the reader. 
\begin{proposition}\label{prop:R}
The map $R^{\mathfrak{b}}_{\mathfrak{j}}$ is an isomorphism of Hilbert spaces. 
\end{proposition}
\begin{remark}
Proposition \ref{prop:R} suggests that our general construction is a proper generalization of the Example of ideal $\langle z_1^2\rangle$ in the Introduction. 
\end{remark}

We consider the trivial vector bundle $E^\mathfrak{b}_{\mathfrak{j}}:=\complex^{(b^1+1)\cdots(b^q+1)}\times \ball_{\mathfrak{j}}$ over $\ball_{\mathfrak{j}}$. The hermitian structure on $E^{\mathfrak{b}}_{\mathfrak{j}}$ is defined as follows. We choose the standard basis of $\{e_{\mathfrak{i}}\}_{i^1\leq b^1, \cdots, i^q\leq b^q}$ of $\complex ^{(b^1+1)\cdots (b^q+1)}$. The Hermitian metric on $E^{\mathfrak{b}}_{\mathfrak{j}}$ at $z\in \ball_{\mathfrak{j}}$ is 
\[
\langle e_{\mathfrak{i}}, e_{\mathfrak{i}'}\rangle_{z}=\delta_{\mathfrak{i}, \mathfrak{i}'}(1-|z|^2)^{q+|\mathfrak{i}|}. 
\]
It is not hard to see that the Hilbert space $\widetilde{\mathcal{H}}_{\mathfrak{j}}^{\mathfrak{b}}$ can be identified with the Bergman space of $L^2$-holomorphic sections of the bundle  $E^{\mathfrak{b}}_{\mathfrak{j}}$. We consider the Toeplitz algebra $\mathfrak{T}(E^{\mathfrak{b}}_{\mathfrak{j}})$ generated by matrix valued Toeplitz operators on the Bergman space of $L^2$-holomorphic sections. Under the isomorphism $R^{\mathfrak{b}}_{\mathfrak{j}}$, one can easily identify the Toeplitz algebra $\mathfrak{T}^{\mathfrak{b}}_{\mathfrak{j}}$ generated by  $T^{\mathfrak{j}, \mathfrak{b}}_{z_i}$, $i=1, \cdots, m$ on $\mathcal{H}_{\mathfrak{j}}^{\mathfrak{b}}$, with the Toeplitz algebra $\mathfrak{T}(E^{\mathfrak{b}}_{\mathfrak{j}})$ on $\widetilde{\mathcal{H}}_{\mathfrak{j}}^{\mathfrak{b}}$. 

\section{Resolutions of Monomial Ideals}\label{sec:resolution}

In this section, we present the proof of the main theorem (Theorem \ref{thm:main-intro}) of this article. In the first three subsections, we generalize the discussion in Section \ref{sec:generalized-Bergman} to construct an exact sequence of Hilbert $A$-modules associated to $k$ boxes in $\mathbb{N}^m$. And in Section \ref{subsec:proof}, we apply our construction to prove Theorem \ref{thm:main-intro}.  
\subsection{Construction of the Hilbert $\complex[z_1,\cdots, z_m]$-modules} \label{subsec:hilbert-mod}
Let $\mathcal{B}_{\mathfrak{j}_1}^{\mathfrak{b}_1},\cdots, \mathcal{B}_{\mathfrak{j}_k}^{\mathfrak{b}_k}$ be $k$ number of boxes in $\mathbb{N}^m$. We start with the following easy property of boxes in $\mathbb{N}^m$.
\begin{lemma}\label{prop:inter-box} Intersections of boxes in $\mathbb{N}^m$ are again boxes. 
\end{lemma}
\begin{proof}
Let us consider two boxes $\mathcal{B}^{\mathfrak{b}_1}_{\mathfrak{j}_1}$ and $\mathcal{B}^{\mathfrak{b}_2}_{\mathfrak{j}_2}$. Let $\mathfrak{j}_{12}$ be the union of $\mathfrak{j}_1$ and $\mathfrak{j}_2$. Suppose that there are $q_{12}$ elements in $\mathfrak{j}_{12}$. Then $\mathfrak{j}_{12}$ can be viewed as an element of $S_{q_{12}}(m)$. We write $\mathfrak{j}_{12}$ as $\mathfrak{j}_{12}:=(j_{12}^1, \cdots, j^{q_{12}}_{12})$. Define $\mathfrak{b}_{12}$ to be an element $\mathfrak{b}_{12}:=(b^1_{12}, \cdots, b^{q_{12}}_{12})$ in $\mathbb{N}^{q_{12}}$ by 
\[
b^k_{12}:=\left\{\begin{array}{ll}\operatorname{min}(b_1^{s}, b_2^{{s'}}),& j_{12}^k=j_1^s=j_2^{s'}\\
						b^{s}_1,& j_{12}^k=j_1^s\notin \mathfrak{j}_1\cap \mathfrak{j}_2 \\
						b^{{s'}}_2,& j_{12}^k=j_2^{s'}\notin \mathfrak{j}_1\cap \mathfrak{j}_2.
		        \end{array}
		\right.
\] 
It is easy to check that the intersection of $\mathcal{B}^{\mathfrak{b}_1}_{\mathfrak{j}_1}$ and $\mathcal{B}^{\mathfrak{b}_2}_{\mathfrak{j}_2}$ is $\mathcal{B}^{\mathfrak{b}_{12}}_{\mathfrak{j}_{12}}$. The general case of the lemma can be proved by induction from the above proof for two boxes. 
\end{proof}

 For any subset $I\subset \{ 1, \cdots, k\}$, we use $\mathcal{B}_{\mathfrak{j}_I}^{\mathfrak{b}_I}$ to denote the following intersection, 
\[
\mathcal{B}_{\mathfrak{j}_I}^{\mathfrak{b}_I}:=\bigcap\limits_{i\in I} \mathcal{B}_{\mathfrak{j}_i}^{\mathfrak{b}_i}. 
\] 
For each box $\mathcal{B}_{\mathfrak{j}_I}^{\mathfrak{b}_I}$, we consider the corresponding Hilbert $A=\complex[z_1, \cdots, z_m]$ module $\mathcal{H}_{\mathfrak{j}_I}^{\mathfrak{b}_I}$ as is introduced in Section \ref{sec:generalized-Bergman}. It is not hard to see that subsets of size $q$ in $\{1, \cdots, k\}$ are in 1-1 correspondence with elements in $S_{q}(k)$.  

Given $\mathcal{B}_{\mathfrak{j}_1}^{\mathfrak{b}_1},\cdots, \mathcal{B}_{\mathfrak{j}_k}^{\mathfrak{b}_k}$, for $1\leq q\leq k$, we define a Hilbert  module $\mathcal{A}_q$ as follows. 
\[
\mathcal{A}_q:=\bigoplus\limits_{I\in S_q(k)} \mathcal{H}_{\mathfrak{j}_I}^{\mathfrak{b}_I}.
\]
For convenience, we use $\mathcal{A}_0$ to denote the Bergman space $L^2_a(\ball^m)$.  We remark that every Hilbert space $\mathcal{A}_q$ is equipped with a Hilbert $A$-module structure from the corresponding $A$-module structure on each component  $\mathcal{H}_{\mathfrak{j}_I}^{\mathfrak{b}_I}$. It follows from Proposition \ref{prop:ess-norm-box} that each $\mathcal{A}_q$ is an essentially normal Hilbert $A$-module. 

\subsection{Morphisms}\label{subsec:morphism}
In this subsection, we define the boundary morphism $\Psi_q:\mathcal{A}_q\to \mathcal{A}_{q+1}$ for $q=0, \cdots, k-1$. In the following, we heavily use the expression introduced in Equation (\ref{eq:expression}).  To explain our construction, we start with a few examples with a small number $k$ of boxes $\mathcal{B}_{\mathfrak{j}_1}^{\mathfrak{b}_1},\cdots, \mathcal{B}_{\mathfrak{j}_k}^{\mathfrak{b}_k}$ in $\mathbb{N}^m$. 

When $k=1$, there is only one box $\mathcal{B}_{\mathfrak{j}}^{\mathfrak{b}}$. We have two Hilbert modules $\mathcal{A}_0=L^2_a(\ball^m)$ and $\mathcal{A}_1=\mathcal{H}^{\mathfrak{b}}_{\mathfrak{j}}$.  We define $\Psi_0: \mathcal{A}_0\to \mathcal{A}_1$ as follows. For $X\in \mathcal{A}_0=L^2_a(\ball^m)$, the map $\Psi_0$ maps $X$ to $Y:=\Psi_0(X)\in \mathcal{H}_{\mathfrak{j}} ^{\mathfrak{b}}$ of the following form, 
\[
Y_{n^1\cdots n^m}:=\left\{\begin{array}{ll}X_{n^1\cdots n^m},& \mathfrak{n}\in \mathcal{B}_{\mathfrak{j}}^{\mathfrak{b}}\\ 0,&\text{otherwise}.\end{array}\right. 
\]

When $k=2$, there are two boxes $\mathcal{B}_{\mathfrak{j}_1}^{\mathfrak{b}_1}$ and $\mathcal{B}_{\mathfrak{j}_2}^{\mathfrak{b}_2}$. We denote their intersection by $\mathcal{B}_{\mathfrak{j}_{12}} ^{\mathfrak{b}_{12}}$.  For $X\in \mathcal{A}_0=L^2_a(\ball^m)$, $\Psi_0(X)$ is written as $Y_1+Y_2$, where $Y_1\in \mathcal{H}^{\mathfrak{b}_1}_{\mathfrak{j}_1}$ and $Y_2\in \mathcal{H}^{\mathfrak{b}_2}_{\mathfrak{j}_2}$ are of the following form, 
\[
(Y_1)_\mathfrak{n}:=\left\{\begin{array}{ll} X_{\mathfrak{n}},&\mathfrak{n}\in \mathcal{B}_{\mathfrak{j}_1}^{\mathfrak{b}_1},
\\ 0,&\text{otherwise}, \end{array}\right.,\qquad (Y_2)_{\mathfrak{n}}:=\left\{\begin{array}{ll} X_{\mathfrak{n}},&\mathfrak{n}\in \mathcal{B}_{\mathfrak{j}_2}^{\mathfrak{b}_2},
\\ 0,&\text{otherwise}. \end{array}\right.
\]

For $(X_1, X_2)\in \mathcal{H}^{\mathfrak{b}_1}_{\mathfrak{j}_1}\oplus \mathcal{H}^{\mathfrak{b}_2}_{\mathfrak{j}_2}=\mathcal{A}_1$, define $\Psi_1(X_1, X_2)\in \mathcal{A}_2$ by
\[
\Psi_1(X_1, X_2)_{\mathfrak{n}}:=\left\{\begin{array}{ll} (X_1)_{\mathfrak{n}}-(X_2)_{\mathfrak{n}},& \mathfrak{n}\in \mathcal{B}_{\mathfrak{j}_{12}}^{\mathfrak{b}_{12}},\\ 0,&\text{otherwise}. \end{array}\right.
\]

For a general $k$,  in order to define the morphism $\Psi_q: \mathcal{A}_q\to \mathcal{A}_{q+1}$, we introduce the following maps $f^i_{q+1}: S_{q+1}(k)\to S_{q}(k)$ for $i=1, \cdots, q+1$. An element in $S_{q+1}(k)$ is a subset $I_{q+1}$ of $\{1,\cdots, k\}$ of size $q+1$. The map $f^i_{q+1}(I_{q+1})$ is the subset of $\{1, \cdots, k\}$ of size $q$ by dropping the $i$-th smallest element in $I_{q+1}$. Define $\Psi_q: \mathcal{A}_q\to \mathcal{A}_{q+1}$ by 
\[
\Psi_q(X):=\sum_{I_{q+1}\in S_{q+1}(k)} Y^{I_{q+1}},\ Y^{I_{q+1}}\in \mathcal{H}_{\mathfrak{j}_{I_{q+1}}}^{\mathfrak{b}_{I_{q+1}}},
\]
for $X=\sum_{J_k\in S_q(k)} X^{J_k}$ with $X^{J_k}\in  \mathcal{H}_{\mathfrak{j}_{J_{k}}}^{\mathfrak{b}_{J_{k}}}$. The function $Y^{I_{q+1}}\in \mathcal{H}_{\mathfrak{j}_{I_{q+1}}}^{\mathfrak{b}_{I_{q+1}}}$ is defined by 
\[
(Y^{I_{q+1}})_{\mathfrak{n}}:=\left\{\begin{array}{ll}\sum_{i=1}^{q+1}(-1)^{i-1} (X^{f^i_{q+1}(I_{q+1})})_{\mathfrak{n}},&\mathfrak{n}\in \mathcal{B}_{\mathfrak{j}_{I_{q+1}}}^{\mathfrak{b}_{I_{q+1}} },\\
										0,&\text{otherwise}.
										\end{array}\right.
\]

\begin{remark}
\label{rmk:geometry} Similar to the explanation in Section \ref{subsec:geometry}, the Hilbert module $\mathcal{A}_i$, $i=1,\cdots, k$, can be identified with  the Bergman space of $L^2$-holomorphic sections of a hermitian vector bundle on a disjoint union of subsets of  the unit ball $\ball^m$. Under this identification, the module morphisms $\Psi_i$, $i=0, \cdots, k-1$, can be realized as restriction maps of jets of holomorphic sections to the subsets. Though we are not using this geometric picture heavily in this paper, we would like to mention it to the reader as we believe that such a geometric picture will play a crucial role in the future study about more general ideals. 
\end{remark}
\subsection{Properties of the box resolution}
In this subsection, we study properties of the box resolutions. 

\begin{proposition}\label{prop:bounded}$\forall q\geq 0$, the morphism $\Psi_q: \mathcal{A}_q\to \mathcal{A}_{q+1}$ is bounded. 
\end{proposition}
\begin{proof} We write $X\in \mathcal{A}_q$ as a sum
\[
X=\sum_{I_q \in S_q(k)} X^{I_q},\qquad X^{I_q}\in \mathcal{H}_{\mathfrak{j}_{I_q}} ^{\mathfrak{b}_{I_q}}.
\]
Then by the above definition,  $\Psi_q(X)=\sum_{I'_{q+1}} Y^{I'_{q+1}}$, where $Y^{I'_{q+1}}\in \mathcal{H}_{\mathfrak{j}_{I'_{q+1}}} ^{\mathfrak{b}_{I'_{q+1}}}$ is equal to 
\[
(Y^{I'_{q+1}})_{\mathfrak{n}}:=\left\{\begin{array}{ll}\sum_{i=1}^{q+1}(-1)^{i-1} (X^{f^i_{q+1}(I'_{q+1})})_{\mathfrak{n}},&\mathfrak{n}\in \mathcal{B}_{\mathfrak{j}_{I'_{q+1}}} ^{\mathfrak{b}_{I'_{q+1}}},\\
										0,&\text{otherwise}.
										\end{array}\right.
\]

The norm of $\Psi_q(X)$ is computed as 
\[
\begin{split}
||\Psi_q(X)||^2=\sum_{I'_{q+1}} ||Y^{I'_{q+1}}||^2&=\sum_{I'_{q+1}}\sum_{\mathfrak{n}\in \mathcal{B}_{\mathfrak{j}_{I'_{q+1}}} ^{\mathfrak{b}_{I'_{q+1}}}}|Y^{I'_{q+1}}_{\mathfrak{n}}|^2\\
&=\sum_{I'_{q+1}}\sum_{\mathfrak{n}\in \mathcal{B}_{\mathfrak{j}_{I'_{q+1}}} ^{\mathfrak{b}_{I'_{q+1}}}} |\sum_{i=1}^{q+1}(-1)^{i-1} (X^{f^i_{q+1}(I'_{q+1})})_{\mathfrak{n}}|^2\\
&\text{by the Cauchy-Schwartz inequality}\\
&\leq \sum_{I'_{q+1}}\sum_{\mathfrak{n}\in \mathcal{B}_{\mathfrak{j}_{I'_{q+1}}} ^{\mathfrak{b}_{I'_{q+1}}}} (q+1)|(X^{f^i_{q+1}(I'_{q+1})})_{\mathfrak{n}}|^2\\
\end{split}
\]
\[
\begin{split}
&\text{as }\mathcal{B}_{\mathfrak{j}_{I'_{q+1}}} ^{\mathfrak{b}_{I'_{q+1}}}\subseteq \mathcal{B}_{\mathfrak{j}_{I_{q}}} ^{\mathfrak{b}_{I_q}}\\
&\leq \sum_{I'_{q+1}}\sum_{\mathfrak{n}\in \mathcal{B}_{\mathfrak{j}_{I_{q}}} ^{\mathfrak{b}_{I_q}} } (q+1)|(X^{I_{q}})_{\mathfrak{n}}|^2\\
&\text{as every }I_{q}\ \text{is contained in at most }(k-q)\ \text{number of}\ I'_{q+1}\\
&\leq (k-q)(q+1) \sum_{I\in S_q(k)}\sum_{\mathfrak{n}\in \mathcal{B}_{\mathfrak{j}_{I_{q}}} ^{\mathfrak{b}_{I_q}} } |(X^{I_{q}})_{\mathfrak{n}}|^2\\
&=(k-q)(q+1)||X||^2. 
\end{split}
\]
\end{proof}

\begin{proposition}\label{prop:module-morphism} The map $\Psi_q: \mathcal{A}_q\to \mathcal{A}_{q+1}$ is an $A=\complex[z_1, \cdots, z_m]$-module morphism. 
\end{proposition}
\begin{proof}
For every $I\in S_q(k)$, for $X^I\in \mathcal{H}^{\mathfrak{b}_I} _{\mathfrak{j}_I}$, $\Psi_q(X^I)$ is a sum 
\[
\sum_{1\leq s\leq k, s\notin I} (-1)^{\text{sign}(I, s)} Y^{I\cup \{ s\}},
\]
where $Y^{I\cup \{s\}}\in \mathcal{H}^{\mathfrak{b}_{I\cup \{s\}}} _{\mathfrak{j}_{I\cup \{s\}}}$, and $s$ is the $\alpha$-th smallest number in $I\cup \{s\}$, and sign$(I,s)$=$\alpha-1$, and the function $Y^{I\cup \{s\}}$ has the following form, 
\[
Y^{I\cup \{s\}}_{\mathfrak{n}}=\left\{\begin{array}{ll} X^I_{\mathfrak{n}},& \mathfrak{n}\in \mathcal{B}^{\mathfrak{b}_{I\cup \{s\}}}_{\mathfrak{j}_{I\cup \{s\}}},\\ 
0,&\text{otherwise}.
\end{array}\right.
\]

If $p\in \{1, \cdots, m\}$, the $z_p$ action on $\mathcal{H}_{\mathfrak{j}_I} ^{\mathfrak{b}_I}$ is as follows, 
\[
T^{\mathfrak{j}_I, \mathfrak{b}_I}_{z_p}(X^I)_{n_1\cdots (n_p+1)\cdots n_m}=\left\{\begin{array}{ll}\sqrt{\frac{\omega_0(n_1\cdots (n_p+1)\cdots n_m)}{\omega_0(n_1\cdots n_m)}} X^I_{n_1\cdots n_p\cdots n_m},& p\notin \mathfrak{j}_I,\\
\sqrt{\frac{\omega_0(n_1\cdots (n_p+1)\cdots n_m)}{\omega_0(n_1\cdots n_m)}} X^I_{n_1\cdots n_p\cdots n_m},& p=j^s\in \mathfrak{j}_I,\ \text{and}\ n_p+1\leq b^s,\\
0,&\text{otherwise}.
\end{array}\right.
\]
From this, we observe that the operator $T^{\mathfrak{j}_I, \mathfrak{b}_I}_{z_p}$ preserves the component $\mathcal{H}_{\mathfrak{j}_I} ^{\mathfrak{b}_I}$. Similarly, the $z_p$ action on $\mathcal{H}^{\mathfrak{b}_{I\cup \{s\}}}_{\mathfrak{j}_{I\cup \{s\}}}$ is as follows,
\[
\begin{split}
&T^{\mathfrak{j}_{I\cup\{s\}}, \mathfrak{b}_{I\cup \{s\}}}_{z_p}(Y^{I\cup \{s\}})_{n_1\cdots (n_p+1)\cdots n_m}\\
=&\left\{\begin{array}{ll}\sqrt{\frac{\omega_0(n_1\cdots (n_p+1)\cdots n_m)}{\omega_0(n_1\cdots n_m)}} Y^{I\cup\{s\}}_{n_1\cdots n_p\cdots n_m},& p\notin \mathfrak{j}_I,\ \ p\ne s,\\
\sqrt{\frac{\omega_0(n_1\cdots (n_p+1)\cdots n_m)}{\omega_0(n_1\cdots n_m)}} Y^{I\cup\{s\}}_{n_1\cdots n_p\cdots n_m},& p=j^t\in \mathfrak{j}_{I\cup \{s\}},\  n_p+1\leq b^t,\\
0,&\text{otherwise}.
\end{array}\right.
\end{split}
\]

Using the above definition of $\Psi_q(X^I)$, we can directly check  that on each component $\mathcal{H}^{\mathfrak{b}_{I\cup \{s\}}}_{\mathfrak{j}_{I\cup \{s\}}}$, 
\[
\Big(\Psi_q\big(T^{\mathfrak{j}_I, \mathfrak{b}_I}_{z_p}(X^I)\big)\Big)^{I\cup \{s\}}= T^{\mathfrak{j}_{I\cup \{s\}}, \mathfrak{b}_{I\cup \{s\}}}_{z_p}\Big(\Psi_q\big( X^I\big)^{I\cup \{ s\}}\Big),
\]
which shows that $\Psi_q$ is compatible with the $A$-module structure. 
\end{proof} 

\begin{proposition}\label{prop:kernel} $\operatorname{Im}(\Psi_{q-1})\subseteq \ker(\Psi_q)$.
\end{proposition}
\begin{proof}

For every $I\in S_{q-1}(k)$ and any $X^I\in \mathcal{H}^{\mathfrak{b}_I}_{\mathfrak{j}_I}$, the image of $X^I$ under $\Psi_{q-1}$ is of the form 
\[
\sum_{1\leq s\leq k, s\notin I} (-1)^{\text{sign}(I, s)} Y^{I\cup \{ s\}},
\]
where $Y^{I\cup \{s\}}\in \mathcal{H}^{\mathfrak{b}_{I\cup \{s\}}} _{\mathfrak{j}_{I\cup \{s\}}}$, and $s$ is the $\alpha$-th smallest number in $I\cup \{s\}$, and sign$(I,s)$=$\alpha-1$, and the function $Y^{I\cup \{s\}}$ has the following form, 
\[
Y^{I\cup \{s\}}_{\mathfrak{n}}=\left\{\begin{array}{ll} X^I_{\mathfrak{n}},& \mathfrak{n}\in \mathcal{B}^{\mathfrak{b}_{I\cup \{s\}}}_{\mathfrak{j}_{I\cup \{s\}}},\\ 
0,&\text{otherwise}.
\end{array}\right.
\]

Similarly, the image of $Y^{I\cup \{s\}}$ under the map $\Psi_q$  is of the form 
\[
\sum_{1\leq t\leq k, t\notin I\cup \{s\}} (-1)^{\text{sign}(I\cup \{s\}, t)} Z^{I\cup \{ s,t\}},
\]
where $Z^{I\cup \{s,t\}}\in \mathcal{H}^{\mathfrak{b}_{I\cup \{s,t\}}} _{\mathfrak{j}_{I\cup \{s,t\}}}$, and $t$ is the $\beta$-th smallest number in $I\cup \{s, t\}$, and sign$(I\cup\{s\}, t)$=$\beta-1$, and the function $Z^{I\cup \{s,t\}}$ has the following form, 
\[
Z^{I\cup \{s,t\}}_{\mathfrak{n}}=\left\{\begin{array}{ll} Y^{I\cup \{s\}}_{\mathfrak{n}},& \mathfrak{n}\in \mathcal{B}^{\mathfrak{b}_{I\cup \{s,t\}}}_{\mathfrak{j}_{I\cup \{s,t\}}},\\ 
0,&\text{otherwise}.
\end{array}\right.
\]

Combining the above computation, we have the following expression for $\Psi_q(\Psi_{q-1}(X^I))$,
\[
\begin{split}
\Psi_q(\Psi_{q-1}(X^I))&=\sum_{1\leq s\leq k, s\notin I} (-1)^{\text{sign}(I, s)} \Psi_q(Y^{I\cup \{ s\}})\\
&=\sum_{1\leq s\leq k, s\notin I} (-1)^{\text{sign}(I, s)} \sum_{1\leq t\leq k, t\notin I\cup \{s\}} (-1)^{\text{sign}(I\cup \{s\}, t)} Z^{I\cup \{ s,t\}}\\
&=\sum_{1\leq s\ne t\leq k, s,t\notin I} (-1)^{\text{sign}(I, s)+\text{sign}(I\cup \{s\}, t)} Z^{I\cup \{ s,t\}}\\
&=\sum_{1\leq s<t\leq k, s,t\notin I} \Big((-1)^{\text{sign}(I, s)+\text{sign}(I\cup \{s\}, t)}+(-1)^{\text{sign}(I,t)+\text{sign}(I\cup \{t\}, s)} \Big)Z^{I\cup \{s,t\}}.
\end{split}
\]
When $s<t$, it is not hard to check the following equations
\[
\text{sign}(I, s)=\text{sign}(I\cup \{t\}, s),\qquad \text{sign}(I\cup \{s\}, t)=\text{sign}(I,t)+1.
\]
And we conclude that $\Psi_q(\Psi_{q-1}(X^I))=0$, and complete the proof of this proposition.
\end{proof}

\begin{proposition}\label{prop:exactness} $$\operatorname{Im}(\Psi_{q-1})\supseteq \ker(\Psi_q), q=1, \cdots, k.$$
\end{proposition}
\begin{proof}
We prove the proposition by induction on the number $k$.  

For $k=1$, we consider the map $\Psi_0: \mathcal{A}_0\to \mathcal{A}_1$. With the orthonormal basis, it is not hard to observe that $\mathcal{A}_1$ can be identified with a closed subspace of $\mathcal{A}_0=L^2_a(\ball^m)$, and the map $\Psi_0$ is the corresponding orthogonal projection map. Therefore, $\Psi_0$ is a surjective map. 

Suppose that the following is true
\[
\operatorname{Im}(\Psi_{q-1})\supseteq \ker(\Psi_q), q=1, \cdots, k,
\]
for all $1\leq k<p$.

We look at the case that $k=p$, and separate the proof into two different cases. 

\noindent{\bf Case I: $q=1$}. Consider $X=(X^1, \cdots, X^p)\in \mathcal{A}_1$ such that $X\in \ker{\Psi_1}$. Define a function $\xi\in \mathcal{A}_0=L^2_a(\ball^m)$ as follows.
\[
\xi_{\mathfrak{n}}:=\left\{ \begin{array}{ll} X^s_\mathfrak{n},& \text{there is}\ s\ \text{such that}\ \mathfrak{n}\in \mathcal{B}^{\mathfrak{b}_s} _{\mathfrak{j}_s}, \\ 0,& \text{otherwise}.\end{array}\right.
\]
We observe that if there are $s,t$ such that $\mathfrak{n}$ belongs to both $ \mathcal{B}^{\mathfrak{b}_s} _{\mathfrak{j}_s}$ and $ \mathcal{B}^{\mathfrak{b}_t} _{\mathfrak{j}_t}$, then the $\mathcal{H}^{\mathfrak{b}_{st}}_{\mathfrak{j}_{st}}$ component of $\Psi_1(X)$ is 
\[
X^s_{\mathfrak{n}}-X^t_{\mathfrak{n}}=0,
\]
as $X\in \ker(\Psi_1)$. Hence, the value $\xi_{\mathfrak{n}}$ is independent of the choices of $s$. And therefore $\xi$ is well defined. Furthermore,  $||\xi||^2$ is no more than the sum 
\begin{equation}\label{eq:norm}
||X^1||^2+\cdots +||X^p||^2. 
\end{equation}
Hence, $\xi\in \mathcal{A}_0$ and $\Psi_0(\xi)=X\in \ker(\Psi_1)$.

\noindent{\bf Case II: $2\leq q\leq p-1$.} We consider the following two collections of $p-1$ boxes, 
\begin{enumerate}
\item the first $p-1$ boxes 
\[
\{\mathcal{B}_{\mathfrak{j}_1}^{\mathfrak{b}_1}, \cdots, \mathcal{B}_{\mathfrak{j}_{p-1}}^{\mathfrak{b}_{p-1}}\}.
\]
We follow the construction in Section \ref{subsec:hilbert-mod}-\ref{subsec:morphism} and consider the associated $A$-modules $\mathcal{A}^1_s$ together with the $A$-module morphisms  $\Psi^1_s: \mathcal{A}^1_s\to \mathcal{A}^1_{s+1}$, $s=1,\cdots, p-2$. Set $\mathcal{A}^1_p:=\{0\}$, and $\Psi^1_{p-1}=0$.  
\item the intersections of the first $p-1$ boxes with the last one $\mathcal{B}_{\mathfrak{j}_p}^{\mathfrak{b}_p}$,
\[
\{ \mathcal{B}_{\mathfrak{j}_{1p}}^{\mathfrak{b}_{1p}}, \cdots, \mathcal{B}_{\mathfrak{j}_{p-1 p}}^{\mathfrak{b}_{p-1 p}}\}.
\]
We follow the construction in Section \ref{subsec:hilbert-mod}-\ref{subsec:morphism}  and consider the associated $A$-modules $\mathcal{A}^2_s$ together with the $A$-module morphisms $\Psi^2_s: \mathcal{A}^2_{s}\to \mathcal{A}^2_{s+1}$, $s=1, \cdots, p-2$. Set $\mathcal{A}^2_p:=\{0\}$, and $\Psi^2_{p-1}=0$. 
\end{enumerate}
By the induction assumption, we know that 
\[
\operatorname{Im}(\Psi^1_{q-1})\supseteq \ker(\Psi^1_q),\qquad \operatorname{Im}(\Psi^2_{q-1})\supseteq \ker(\Psi^2_q),\qquad q=1, \cdots, p-1.
\]

We define a map $\Phi_t: \mathcal{A}^1_{t}\to \mathcal{A}^2_{t}$ by 
\[
\Phi_t(X^I)=Y^{I\cup \{p\}},\ I\in S_t(p-1)
\]
where by $Y^{I\cup \{p\}}$ we refer to the component corresponding to the intersection of the boxes $\mathcal{B}_{\mathfrak{j}_{i_1p}}^{\mathfrak{b}_{i_1p}}, \cdots, \mathcal{B}_{\mathfrak{j}_{i_{t}p}}^{\mathfrak{b}_{i_tp}}$ with
\[
Y^{I\cup \{p\}}_{\mathfrak{n}}:=\left\{\begin{array}{ll} (-1)^{t}X^I_{\mathfrak{n}},& \mathfrak{n}\in \mathcal{B}_{\mathfrak{j}_{I\cup\{p\}}} ^{\mathfrak{b}_{I\cup \{p\}}}\\ 0,&\text{otherwise}.\end{array}\right.
\]
Similar to Proposition \ref{prop:module-morphism}, $\Phi_t$ is an $A$-module morphism. We leave the detail to the reader to check. 

With the above construction, we can easily check the following identities. 
\begin{enumerate}
\item $\mathcal{A}_q=\mathcal{A}^1_{q}\oplus \mathcal{A}^2_{q-1}$, for $q=2,\cdots, p$, where $\mathcal{A}^1_p=\{0\}$. 
\item $\Psi_q=\left(\begin{array}{cc} \Psi^1_q&0\\ \Phi_q& \Psi^2_{q-1}\end{array}\right)$, for $q=2, \cdots, p-1$, where $\Psi^1_{p-1}=0$.  
\end{enumerate}

We use the above identifications to prove $\text{Im}(\Psi_{q-1})\supseteq \ker(\Psi_{q})$. The proof consists of the following three cases. 
\[
i)\ q=2,\qquad ii)\ 3\leq q\leq p-2,\qquad iii)\ q=p-1. 
\] 

\noindent{$i)\ q=2$}. 

Suppose $(X_1, X_2)\in \mathcal{A}^1_2\oplus \mathcal{A}^2_1=\mathcal{A}_2$ is in the kernel of the morphism $\Psi_2$. By the above identification of $\Psi_q$, we have 
\[
\Psi^1_2(X_1)=0,\qquad \Phi_2(X_1)+\Psi^2_1(X_2)=0. 
\]
By the induction assumption, $\ker(\Psi^1_2)\subseteq \text{Im}(\Psi^1_1)$. So there exists $Y_1\in \mathcal{A}^1_1$ such that $\Psi^1_1(Y_1)=X_1$.  By Proposition \ref{prop:kernel} for the morphism $\Psi_\bullet$, we have
\[
\begin{split}
(0,0)=\Psi_2\big(\Psi_1(Y_1,0)\big)&=\Psi_2\big(\Psi^1_1(Y_1), \Phi_1(Y_1)\big)=\Big(\Psi^1_2\big(\Psi^1_1(Y_1)\big), \Phi_2\big(\Psi^1_1(Y_1)\big)+\Psi^2_1\big(\Phi_1(Y_1)\big)\Big)\\
&\text{as}\ \Psi^1_1(Y_1)=X_1, \Psi^1_2\big(\Psi^1_1(Y_1)\big)=0\\
&=\Big(0, \Phi_2(X_1)+\Psi^2_1\big(\Phi_1(Y_1)\big) \Big).
\end{split}
\]
Hence, $\Phi_2(X_1)+\Psi^2_1\big(\Phi_1(Y_1)\big)=0$. Consider $X_2'=X_2-\Phi_1(Y_1)$.  We compute 
\[
\Psi^2_1(X_2')=\Psi^2_1(X_2-\Phi_1(Y_1))=\Psi^2_1(X_2)-\Psi^2_1(\Phi_1(Y_1))=\Psi^2_1(X_2)+\Phi_2(X_1)=0,
\] 
as $\Psi_2(X_1, X_2)=(\Psi^1_2(X_1), \Phi_2(X_1)+\Psi^2_1(X_2))=0$. 

Using the property that $\Psi^2_1(X_2')=0$, we construct an element $Y_2\in \mathcal{H}_{\mathfrak{j}_p}^{\mathfrak{b}_p}$ by setting
\[
(Y_2)_{\mathfrak{n}}:=\left\{\begin{array}{ll}({X'_2}^{ip})_{\mathfrak{n}},&\mathfrak{n}\in \mathcal{B}^{\mathfrak{b}_{ip}}_{\mathfrak{j}_{ip}},\ \text{for some }\ i=1,..., p-1,\\ 0,&\text{otherwise}. \end{array}\right.
\]
As $\Psi^2_1(X_2')=0$, the above definition of $Y_2$ is independent of the choices of $i$. It is not hard to check the norm of $Y_2$ is bounded following the similar estimate as Equation (\ref{eq:norm}). Therefore, $Y_2$ is in $\mathcal{H}_{\mathfrak{j}_p}^{\mathfrak{b}_p}\subset L^2_a(\ball^m)$, and $\Psi^2_0(Y_2)=X_2'$. 

In summary, we have constructed an element $(Y_1, Y_2)\in \mathcal{A}_1=\mathcal{A}^1_1\oplus\mathcal{H}_{\mathfrak{j}_p}^{\mathfrak{b}_p}$. And it is not hard to check that 
\[
\Psi_1(Y_1, Y_2)=(\Psi^1_1(Y_1), \Phi_1(Y_1)+\Psi_0^2(Y_2))=(X_1, \Phi_1(Y_1)+X_2')=(X_1, X_2), 
\]
which shows that $(X_1, X_2)\in \text{Im}(\Psi_1).$

\noindent{ii) $3\leq q\leq p-2$}. 

Suppose $(X_1, X_2)\in \mathcal{A}^1_{q}\oplus \mathcal{A}^2_{q-1}=\mathcal{A}_q$ is in the kernel of the morphism $\Psi_q$.  By the above identification of $\Psi_q$, we have 
\[
\Psi^1_q(X_1)=0,\qquad \Phi_q(X_1)+\Psi^2_{q-1}(X_2)=0. 
\]
As $\operatorname{Im}(\Psi^1_{q-1})\supseteq \ker(\Psi^1_q)$, there is $Y_1\in \mathcal{A}^1_{q-1}$ such that $X_1=\Psi^1_{q-1}(Y_1)$. As $\Psi_q(\Psi_{q-1}(Y_1, 0))=0$, 
$\Phi_q(X_1)+\Psi^2_{q-1}(\Phi_{q-1}(Y_1))=0$. Hence, we have
\[
\Psi^2_{q-1}(X_2-\Phi_{q-1}(Y_1))=0. 
\]
As $\operatorname{Im}(\Psi^2_{q-2})\supseteq \ker(\Psi^2_{q-1})$, there exists $Y_2\in \mathcal{A}^2_{q-2}$ such that 
\[
\Psi^2_{q-2}(Y_2)=X_2-\Phi_{q-1}(Y_1). 
\]
Therefore, we have found $(Y_1, Y_2)\in \mathcal{A}_q$ such that 
\[
\Psi_{q-1}(Y_1, Y_2)=\big(\Psi^1_{q-1}(Y_1), \Phi_{q-1}(Y_1)+\Psi_{q-2}(Y_2)\big)=(X_1, X_2). 
\]

\noindent{$iii)\ q=p-1$.}

We notice that $\mathcal{A}_{p}$ is the same as $\mathcal{A}^2_{p-1}$. As the map $\Psi^2_{p-3}: \mathcal{A}^2_{p-2}\to \mathcal{A}^2_{p-1}$ is surjective, it follows that the map $\Psi_{p-1}: \mathcal{A}_{p-1}=\mathcal{A}^1_{p-1}\oplus \mathcal{A}^2_{p-2}\to \mathcal{A}_{p}=\mathcal{A}^2_{p-1}$ is surjective. 

\end{proof}
\subsection{Proof of Theorem \ref{thm:main-intro}}\label{subsec:proof}

In this subsection, we prove the main theorem of the paper. We assume that $I$ is an ideal of the polynomial algebra $\complex[z_1, \cdots, z_m]$ generated by monomials $z^{\alpha_i}:=z_1^{\alpha^1_i}\cdots z_{m}^{\alpha^m_i}$, for $\alpha_i\in \mathbb{N}^m$, $i=1,\cdots, l$. 

Following \cite[Theorem 1.1.2]{he-hi:monomial}, monomials inside the ideal $I$ form a linear basis of the ideal $I$ over $\complex$. We consider the lattice $\mathbb{N}^m$ and the subset   $C(I)$ consisting of exponents of monomials that do not belong to $I$. According to \cite[Proposition 1.1.5]{he-hi:monomial}, a monomial $v$ belongs to $I$ if and only if there is a monomial $w$ such that $v=wz^{\alpha_i}$ for some $i=1, \cdots, l$.  Therefore, the monomial $z_1^{n^1}\cdots z_m^{n^m}$ does not belong to $I$ if and only for any $i=1,\cdots, l$, $z^{\alpha_i}$ is not a factor of $z_1^{n^1}\cdots z_m^{n^m}$. Equivalently, $z_1^{n^1}\cdots z_m^{n^m}$ does not belong to $I$ if and only if for every $i$, there is $s_i$ such that $n^{s_i}< \alpha_i^{s_i}$. 

Consider the finite collection $S(\alpha_1, \cdots, \alpha_l)$ of $l$-tuple of natural numbers $\mathfrak{s}=(s_1, \cdots, s_l)$ such that $1\leq s_i\leq m$. For each $\mathfrak{s}$, let $\mathfrak{j}_{\mathfrak{s}}\subseteq \{1,\cdots,m\}$ be the subset consisting of those numbers appearing in the array $(s_1, \cdots, s_l)$. For every $k\in \mathfrak{j}_{\mathfrak{s}}$, let $b_k$ be the minimum of all $\alpha_i^{s_i}$ such that $s_i=k$ for $i=1, \cdots, l$.  From the above conditions on monomials not in $I$, we conclude that $C(I)$ is the union of all  boxes $\mathcal{B}_{\mathfrak{s}}$. We refer the reader to \cite[Section 9.2, Theorem 3]{co-li-osh:ideals} for a related discussion. We conclude from the above discussion that a polynomial $f$ belongs to $I$ if and only if $f$ has no nonzero component in any of the boxes $\mathcal{B}_{\mathfrak{j}_{\mathfrak{s}}}^{\mathfrak{b}_{\mathfrak{s}}}$ for any $\mathfrak{s}\in S(\alpha_1, \cdots, \alpha_l)$. Let $\mathcal{B}_{\mathfrak{j}_1}^{\mathfrak{b}_1},\cdots, \mathcal{B}_{\mathfrak{j}_k}^{\mathfrak{b}_k}$ be the collection of nonempty boxes in 
$\{\mathcal{B}_{\mathfrak{j}_{\mathfrak{s}}}^{\mathfrak{b}_{\mathfrak{s}}}: \mathfrak{s}\in S(\alpha_1, \cdots, \alpha_l)\}$ associated to the ideal $I$. 

Associated to the above collection of boxes $\mathcal{B}_{\mathfrak{j}_1}^{\mathfrak{b}_1},\cdots, \mathcal{B}_{\mathfrak{j}_k}^{\mathfrak{b}_k}$, we apply the constructions in Section \ref{subsec:hilbert-mod}-\ref{subsec:morphism} to construct a sequence of Hilbert $A$-modules $\mathcal{A}_0, \cdots, \mathcal{A}_k$ together with module morphisms $\Psi_q: \mathcal{A}_q \to \mathcal{A}_{q+1}$, i.e. 
\[
\mathcal{A}_0=L^2_a(\ball^m)\stackrel{\Psi_0}{\longrightarrow} \mathcal{A}_1\stackrel{\Psi_1}{\longrightarrow}\cdots \stackrel{\Psi_{k-1}}{\longrightarrow }\mathcal{A}_k\longrightarrow 0.
\]
Proposition \ref{prop:ess-norm-box} shows that each $\mathcal{A}_q$ ($q=0, 1, ..., k$) is a Hilbert $A=\complex[z_1, \cdots, z_m]$-module. Proposition \ref{prop:bounded}-\ref{prop:module-morphism} show that $\Psi_q$ ($q=0,\cdots, k-1$) is a bounded module morphism. Proposition \ref{prop:kernel}-\ref{prop:exactness} show that the above sequence is exact at $\mathcal{A}_q$ for $q=1, \cdots, k$. 

We are left to prove that the kernel of the morphism $\Psi_0$ is the completion of $I$ in $L^2_a(\ball^m)$, i.e.
\[
\overline{I}=\ker(\Psi_0). 
\]
By the above discussion, if $f\in I$, then $f$ has no nonzero component in any of the boxes $\mathcal{B}_{\mathfrak{j}_{\mathfrak{s}}}^{\mathfrak{b}_{\mathfrak{s}}}$ for any $\mathfrak{s}\in S(\alpha_1, \cdots, \alpha_l)$. This shows that $f\in \ker(\Psi_0)$. Therefore, $I$ and its closure $\overline{I}$ are contained inside $\ker(\Psi_0)$. 

Suppose that $f$ is in the kernel $\ker(\Psi_0)$. Write $f$ in terms of the orthonormal basis, 
\[
f=\sum_{\mathfrak{n}\in \mathbb{N}^m} f_{\mathfrak{n}}z^{\mathfrak{n}}.
\]
As $\Psi_0(f)=0$, by the definition of $\Psi_0$, for any $s=1, \cdots, k$, and any $\mathfrak{n}\in \mathcal{B}_{\mathfrak{j}_s}^{\mathfrak{b}_s}$, $f_{\mathfrak{n}}=0$. For any positive integer $M$, let $f_M$ be the truncation of the above expansion of $f$ by requiring $n^1, \cdots, n^m<M$, i.e. 
\[
f_M:=\sum_{\mathfrak{n}\in \mathbb{N}^m, n^1<M, \cdots, n^m<M} f_{\mathfrak{n}}z^{\mathfrak{n}}.
\]
It is not hard to see that $f_M$ is a polynomial, and has no component in the  boxes $\mathcal{B}_{\mathfrak{j}_1}^{\mathfrak{b}_1},\cdots, \mathcal{B}_{\mathfrak{j}_k}^{\mathfrak{b}_k}$. By the construction of the boxes $\mathcal{B}_{\mathfrak{j}_1}^{\mathfrak{b}_1},\cdots, \mathcal{B}_{\mathfrak{j}_k}^{\mathfrak{b}_k}$,  $f_M$ belongs to the ideal $I$. As $M\to \infty$, $f_M$ converges to $f$ in $L^2_a(\ball^m)$. Hence, we have shown that $f$ belongs to the closure $\overline{I}$, and $\ker(\Psi_0)$ is a subset of $\overline{I}$. So we conclude that $\overline{I}=\ker(\Psi_0)$. 

In summary, we have completed the proof of Theorem \ref{thm:main-intro} for general monomial ideals. 

\subsection{K-homology class}\label{subsec:k-hom} 

As a corollary of Theorem \ref{thm:main-intro}, by applying \cite[Theorem 1]{do:module}, we can conclude that the closure $\overline{I}$ of the ideal $I$ and the quotient $Q_I:=L^2_a(\mathbb{B}^m)/\overline{I}$ are both essentially normal Hilbert modules. 

Let $\mathfrak{T}(\overline{I})$ (and $\mathfrak{T}(Q_{I})$) be  the unital $C^*$-algebra generated by Teoplitz operators on the module $\overline{I}$ (and the quotient module $Q_I$). We would like to discuss properties of the $K$-homology class associated to the following Toeplitz extension,
\[
0\longrightarrow \mathcal{K}\longrightarrow \mathfrak{T}(Q_I)\longrightarrow C(\sigma^e_I)\longrightarrow 0,
\]
where $\sigma^e_I$ is the essential spectrum space of the algebra $\mathfrak{T}(Q_I)$ and $\mathcal{K}$ is the algebra of compact operators.

By Theorem \ref{thm:main-intro}, for $i=1, \cdots, k$, we introduce the following closed subspace of $\mathcal{A}_i$,
\[
\mathcal{A}_i^{-}:=\operatorname{Im}(\Psi_{i-1})=\operatorname{ker}(\Psi_i). 
\]
As $\Psi_{k-1}$ is surjective, $\mathcal{A}_k^-=\mathcal{A}_k$. 

As $\Psi_i: \mathcal{A}_i\to \mathcal{A}_{i+1}$ is a morphism of $A=\complex[z_1, \cdots, z_m]$-modules, the kernel $\mathcal{A}_i^{-}=\operatorname{ker}(\Psi_i)$ is naturally an $A$-module. Furthermore, we have the following exact sequence of Hilbert $A$-modules,
\[
0\longrightarrow \mathcal{A}_{i}^-\longrightarrow \mathcal{A}_i\longrightarrow \mathcal{A}_{i+1}^{-}\longrightarrow 0,\ i=1, \cdots, k-1, 
\]
where the first map is the inclusion, and the second map is $\Psi_i$. 
\begin{lemma}\label{lem:sub-module}  $\forall i=1, \cdots, k$, the $A$-module $\mathcal{A}_i^{-}$ is essentially normal, and therefore the quotient module 
$Q_i:=\mathcal{A}_i/\mathcal{A}_{i}^-$ is also essentially normal. 
\end{lemma}
\begin{proof}
When $i=k-1$, as $\Psi_{k-1}$ is surjective, we have the following short exact sequence, 
\[
0\longrightarrow \mathcal{A}_{k-1}^-\longrightarrow \mathcal{A}_{k-1}\longrightarrow \mathcal{A}_{k}\longrightarrow 0. 
\]

As is explained in Section \ref{subsec:hilbert-mod}, both $\mathcal{A}_{k-1}$ and $\mathcal{A}_k$ are essentially normal $A$-modules. It follows from \cite[Theorem 1]{do:module} that $\mathcal{A}_{k-1}^-$ is an essentially normal $A$-module. 

Repeating the above arguments to the exact sequence 
\[
0\longrightarrow \mathcal{A}_{k-2}^-\longrightarrow \mathcal{A}_{k-1}\longrightarrow \mathcal{A}_{k-1}^{-}\longrightarrow 0,
\]
we conclude that $\mathcal{A}_{k-2}^-$ is an essentially normal $A$-module. Similarly, the iterated arguments show that every $\mathcal{A}_i^{-}$ is an essentially normal $A$-module. 
\end{proof}

Modeled by the above short exact sequence of essentially normal Hilbert $A$-modules $\mathcal{A}_i^-$ and $\mathcal{A}_i$, we prove the following property.
 
\begin{proposition}\label{prop:isometry} Let $M_1$, $M_2$, $M_3$ be essentially normal Hilbert $A$-modules, and $W_1:M_1\to M_2$, $W_2:M_2\to M_3$ be morphisms of Hilbert $A$-modules satisfying the following short exact sequence
\[
0\longrightarrow M_1\stackrel{W_1}{\longrightarrow} M_2\stackrel{W_2}{\longrightarrow} M_3\longrightarrow 0.
\]
Let $\overline{\ball}^m$ be the closure of the unit ball in $\complex^m$, and $\alpha_i: C(\overline{\ball}^m)\to \mathcal{C}(M_i):=B(M_i)/\mathcal{K}(M_i)$, $i=1, 2,3$ be the induced representation of $C(\overline{\ball}^m)$ on the Calkin algebra $\mathcal{C}(M_i)$.  There are partial co-isometry operators $U: M_2\to M_1$ and $V: M_2\to M_3$ such that 
\[
UU^*=I=VV^*,\ UV^*=0=VU^*,\ U^*U+V^*V=I,  
\] 
and commute with $A$-module structures up to compact operators, i.e. 
\[
U\alpha_2 U^*=\alpha_1,\qquad V\alpha_2V^*=\alpha_3.  
\]
\end{proposition}

\begin{proof} As $W_2$ is surjective, $W_2 W_2^*$ is positive definite. Let $W_2=A_3 V$ be the polar decomposition of the morphism $W_2$ such that $A_3$ is a positive definite operator on $M_3$, s.t. $A_3=(W_2W_2^*)^{\frac{1}{2}}$, and $V$ is a partial coisometry, i.e. $VV^*=I$.   

As $W_2$ is an $A$-module morphism, for any $f\in A=\complex[z_1, \cdots, z_m]$, we have
\[
A_3VT^2_f=W_2 T^2_f=T^3_f W_2=T^3_f A_3V,
\]
where $T^2_f$ and $T^3_f$ are the Toeplitz operators on $M_2$ and $M_3$ associated to $f$. As $M_2$ and $M_3$ are essentially normal, $T^2_f$ and $T^3_f$ are both normal in the respective Calkin algebras. It follows from the Fuglede-Putnam theorem that the following equation holds modulo compact operators
\[
A_3V(T^2_f)^*=W_2 (T^2_f)^*=(T^3_f)^* W_2=(T^3_f)^*A_3V. 
\]
Taking the adjoint of both sides of the above equation, we reach 
\[
T^2_f V^* A_3=V^* A_3 T^3_f. 
\]

Multiplying $A_3V$ to the left of each term in the above equation, with the property $VV^*=I$, we have 
\[
A_3 V T^2_f V^* A_3=A_3VV^*A_3 T^3_f=A_3^2 T^3_f. 
\]
As 
\[
A_3VT^2_f=T^3_f A_3V, 
\]
we conclude from the above equation that modulo compact operators
\[
A_3 V T^2_f V^* A_3=T^3_f A_3 VV^*A_3=T^3_f A_3^2=A_3^2 T^3_f. 
\]
As $A_3$ is positive definite, it is safe to conclude that modulo compact operators 
\[
T^3_f A_3=A_3 T^3_f. 
\]
The above commutativity plus the equation $A_3VT^2_f=T^3_f A_3V$ gives the following identity, modulo compact operators, 
\[
VT^2_f=T^3_f V. 
\]
The property of $VV^*=I$ confirms that modulo compact operators
\[
VT^2_f V^*=T^3_f, 
\]
which is exactly 
\[
V \alpha_2 V^*=\alpha_3. 
\]

As  $W_1: M_1\to M_2$ is injective, $W_1^* W_1$ is positive definite. Let $W_1=WA_1$ be the polar decomposition of the operator $W_1$, where $A_1=(W_1^*W_1)^{\frac{1}{2}}$ and $W: M_1\to M_2$ is a partial isometry, i.e. $W^*W=I$. A similar argument as above for $W_2$ show that modulo compact operators, for any $f\in A=\complex[z_1, \cdots, z_m]$, 
\[
A_1T_f^1=T^1_f A_1,
\]
and 
\[
W^* T^2_f  W=T^1_f.
\]
If we set $U=W^*$, then $UT^2_f U^*=T^1_f$, and $UU^*=I$, which shows 
\[
U\alpha_2U^*=\alpha_1. 
\] 

As $W_2W_1=A_3VU^* A_1=0 $, $VU^*=0$ follows from the invertibility of $A_1$ and $A_3$. Therefore, $U^*U$ and $V^*V$ are commuting two orthogonal projections on $M_2$. To prove that their sum is the identity operator, it is sufficient to prove that the kernel of their sum is trivial. If $\xi\in M_2$ satisfies $U^*U\xi+V^*V\xi=0$, then $U^*U\xi=V^*V\xi=0$, and  $U\xi=V\xi=0$.  Then $W_2\xi=A_3V\xi=0$, and $W_1^*\xi=A_1U\xi=0$. As $W_2\xi=0$, $\xi$ belongs to the kernel of $W_2$, and the exactness of the morphisms  shows that there is $\eta\in W_1$ such that $W_1\eta=\xi$. As $W_1^*\xi=0$, $W_1^*W_1\eta=0$, and $\xi=W_1\eta=0$. Hence the kernel of $U^*U+V^*V$ is trivial, and $U^*U+V^*V=I$. 
\end{proof}

Let $\mathfrak{T}(M_i)$ be the unital $C^*$-algebra generated by Toeplitz operators on $M_i$, and $\sigma^e_{i}$ be the associated essential spectrum space. The following property is a quick corollary of Proposition \ref{prop:isometry}. 

\begin{corollary}\label{cor:k-hom}Under the same assumption as Proposition \ref{prop:isometry}, the following equation holds in $K_1(\sigma^e_2)$, 
\[
[\alpha_2]=[\alpha_1]+[\alpha_3],
\]
where $[\alpha_1]$ and $[\alpha_3]$ are identified as classes in $K_1(\sigma^e_2)$ by the partial coisometry operators $U$ and $V$ introduced in Proposition \ref{prop:isometry}. 
\end{corollary}
\begin{proof}
This follows the property for $K$-homology associated to the three representations $\alpha_2, \alpha_1=U\alpha_2U^*, \alpha_3=V\alpha_2V^*: C(\sigma^e_2)\to \mathcal{C}(M_2)$ with $UU^*=I=VV^*$, $UV^*=0=VU^*$, and $U^*U+V^*V=I$. 
\end{proof}

We are now ready to apply Proposition \ref{prop:isometry} and Corollary \ref{cor:k-hom} to the exact sequence constructed in Theorem \ref{thm:main-intro}. Let $\mathfrak{T}(\mathcal{A}_i)$ (and $\mathfrak{T}(\mathcal{A}_i^-)$) be the unital $C^*$-algebra generated by Toeplitz operators on $\mathcal{A}_i$ (and $\mathcal{A}_i^-$), and $\sigma^e_{i}$ (and $\sigma^e_{i-}$) be the associated essential spectrum space, and $\alpha_i$ (and $\alpha_i^-$) be the associated representation of $C(\sigma^e_i)$ (and $C(\sigma^e_{i-})$)  into the Calkin algebra $\mathcal{C}(\mathcal{A}_i)$ (and $\mathcal{C}(\mathcal{A}_i^-)$). 
\begin{theorem}\label{thm:k-hom} (Theorem \ref{thm:k-hom-intro}) The $K$-homology class associated to $[\mathfrak{T}(Q_I)]$ is the same as $[\alpha_1^-]$, and in $K_1(\sigma^e_1\cup \cdots \cup \sigma^e_k)$, 
\[
[\mathfrak{T}(Q_I)]=[\alpha_1]-[\alpha_2]+\cdots+(-1)^{k-1}[\alpha_{k}].
\]
\end{theorem}
\begin{proof}
We apply Corollary \ref{cor:k-hom} to the following short exact sequence of essentially normal Hilbert $A$-modules
\[
0\longrightarrow \mathcal{A}_{i}^-\longrightarrow \mathcal{A}_i\longrightarrow \mathcal{A}^-_{i+1}\longrightarrow 0,\ i=1, \cdots, k-1. 
\]
We have the following equation in $K_1(\sigma_i)$,
\[
[\alpha_i]=[\alpha_i^-]+[\alpha_{i+1}^-]. 
\]

When $i=k-1$, $\mathcal{A}_{k}^-=\mathcal{A}_{k}$, and in $K_1(\sigma^e_{k-1})$,
\[
[\alpha_{k-1}]=[\alpha_{k-1}^-]+[\alpha_k]. 
\]

Similarly, for $i=k-2$, in $K_1(\sigma^e_{k-2})$, 
\[
[\alpha_{k-2}]=[\alpha_{k-2}^-]+[\alpha_{k-1}^-].
\]

Combining the above two equations on $K$-homology groups, we conclude that in $K_1(\sigma^e_{k-1}\cup \sigma^e_{k-2})$, 
\[
[\alpha_{k-1}]+[\alpha_{k-2}^-]=[\alpha_k]+[\alpha_{k-2}],
\]
by pushing forward the respective equations in $K_1(\sigma^e_{k-1})$ and $K_1(\sigma^e_{k-2})$ into the ones in $K_1(\sigma^e_{k-1}\cup \sigma^e_{k-2})$ via the natural inclusion maps $\sigma^e_{k-1}, \sigma^e_{k-2}\hookrightarrow \sigma^e_{k-1}\cup \sigma^e_{k-2} $. 

Repeating the above arguments inductively, we conclude that in $K_1(\sigma_1^e\cup \cdots \cup \sigma_k^e)$, 
\[
[\alpha_{1}^-]=[\alpha_1]-[\alpha_2]+\cdots+(-1)^{k-1}[\alpha_{k}]. 
\] 

By the short exact sequence of essentially normal Hilbert $A$-modules,
\[
0\longrightarrow \overline{I}\longrightarrow L^2_a(\ball^m)\longrightarrow \mathcal{A}_1^-\longrightarrow 0,
\]
we conclude that there is a natural $A$-module isomorphism between the quotient Hilbert modules $Q_I=L^2_a(\ball^m)/\overline{I}$ and $\mathcal{A}_1^-$. 
We conclude from \cite[Proposition 4.4]{do-ta-yu:grr} that they must define the same $K$-homology class. Therefore, we conclude that in $K_1(\sigma_1^e\cup \cdots \cup \sigma_k^e)$, 
\[
[\mathfrak{T}(Q_I)]=[\alpha_1]-[\alpha_2]+\cdots+(-1)^{k-1}[\alpha_{k}]. 
\] 

\end{proof}
\section{Examples}\label{sec:example}
In this section, we explain our construction of the boxes in examples. 
\subsection{Ideal $I=\langle z_1^2 z_2^2\rangle\subset \complex[z_1, z_2]$}
The exponents of monomials in the ideal $I=\langle z_1^2 z_2^2\rangle$ comprises the region $\{(n^1, n^2)| n^1, n^2\geq 2\}$. In this example, there is only one $\alpha=(2,2)$. There are two boxes associated to the ideal, $\mathcal{B}_{\mathfrak{j}_1}^{\mathfrak{b}_1}:=\{(n^1, n^2)| n^1\leq 1\}$, $\mathcal{B}_{\mathfrak{j}_2}^{\mathfrak{b}_2}:=\{(n^1, n^2)|n^2\leq 1\}$. The intersection $\mathcal{B}_{\mathfrak{j}_{12}}^{\mathfrak{b}_{12}} $ of $\mathcal{B}_{\mathfrak{j}_1}^{\mathfrak{b}_1}$ and $\mathcal{B}_{\mathfrak{j}_2}^{\mathfrak{b}_2}$ is of the form $\{(n^1, n^2)| n^1, n^2\leq 1\}$. 

We have following subsets $\ball_{\mathfrak{j}_1}:=\{(0, z_2)|\ |z_2|<1\}$, $\ball_{\mathfrak{j}_2}:=\{(z_1, 0)|\ |z_1|<1\}$. 

The Hilbert module $\mathcal{A}_1$ is the direct sum of two submodules $\mathcal{A}^1_1$ and $\mathcal{A}^2_1$, where $\mathcal{A}^1_1$ is the closed subspace of $L^2_a(\ball^2)$ spanned by $\{ z_2^n, z_1z_2^n\}_{n\in \mathbb{N}}$, and $\mathcal{A}^2_1$ is the closed subspace of $L^2_a(\ball^2)$ spanned by $\{z_1^n, z_1^n z_2\}_{n\in \mathbb{N}}$. 

The Hilbert $A$-module on $\mathcal{A}_2$ is the subspace of $L^2_a(\ball^2)$ spanned by $\{1, z_1, z_2, z_1z_2\}$.  It is easy to see that $\mathcal{A}_2=\mathcal{A}^1_1\cap \mathcal{A}^2_1$. 
\subsection{Ideal $I=\left\langle z_{1}^{p}z_{2}^{q},z_{1}^{r}z_{2}^{s}\right\rangle\subset \complex[z_1,z_2]$}
We consider the ideal $I=\left\langle z_{1}^{p}z_{2}^{q},z_{1}^{r}z_{2}^{s}\right\rangle\subset \complex[z_1,z_2]$, $p,q,r,s\geq 0$. We explain our construction of boxes and the associated Hilbert modules in Theorem \ref{thm:main-intro} in this example. 

The exponents of monomials in $I$ belong to the white region of the Figure (\ref{fig:33}) marked by $I$. The subset $C(I)\subset \mathbb{N}^2$ consisting of exponents of monomials not in the ideal $I$ is the blue region in Figure (\ref{fig:33}). In this example, $\alpha_1=(r,s), \alpha_2=(p,q)$,  $S(\alpha_1, \alpha_2)$ consists of 4 arrays $(1,1), (1,2), (2,1), (2,2) $.  The boxes associated to these arrays are described below. 
\begin{enumerate}
\item For $\mathfrak{s}=(1,1)$, the box $\mathcal{B}_{\mathfrak{j}_{11}}^{\mathfrak{b}_{11}}$ is $\{(n^1, n^2)| n^1<r\}$;
\item For $\mathfrak{s}=(1,2)$, the box $\mathcal{B}_{\mathfrak{j}_{12}}^{\mathfrak{b}_{12}}=\{(n^1, n^2)|n^1<p, n^2<s\}$; 
\item For $\mathfrak{s}=(2,1)$, the box $\mathcal{B}_{\mathfrak{j}_{21}}^{\mathfrak{b}_{21}}=\{(n^1, n^2)| n^1<r, n^2<q\}$;
\item For $\mathfrak{s}=(2,2)$, the box $\mathcal{B}_{\mathfrak{j}_{22}}^{\mathfrak{b}_{22}}=\{(n^1,n^2)| n^2<q\}$.
\end{enumerate}
\begin{figure}
\centerline{
\includegraphics[scale=0.4]{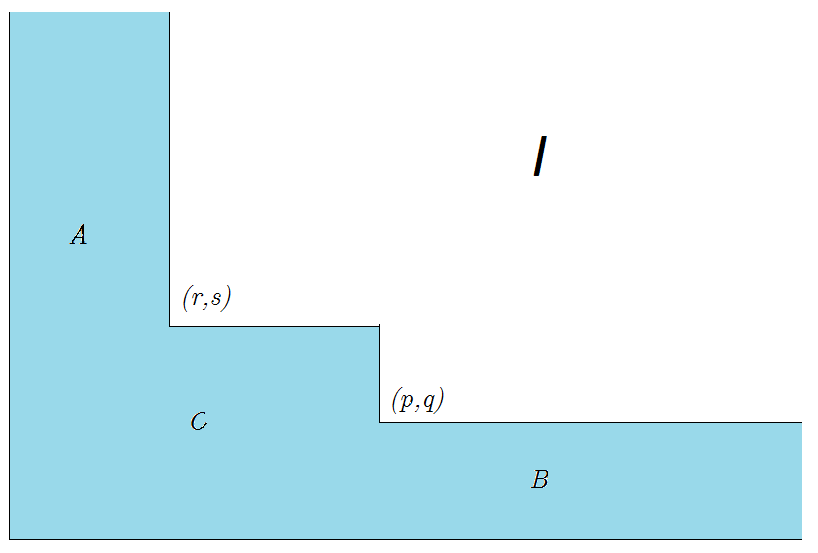}}
\caption
{
\label{fig:33}
{Staircase diagram corresponding to $I=\langle z_{1}^{p}z_{2}^{q},z_{1}^{r}z_{2}^{s}\rangle$.}
}
\end{figure}
In the Figure (\ref{fig:33}), the box $\mathcal{B}_{\mathfrak{j}_{11}}^{\mathfrak{b}_{11}}$ is marked as region $A$, and  the box $\mathcal{B}_{\mathfrak{j}_{12}}^{\mathfrak{b}_{12}}$ is marked as region $C$, and the box $\mathcal{B}_{\mathfrak{j}_{22}}^{\mathfrak{b}_{22}}$ is marked as region $B$. As the box $\mathcal{B}_{\mathfrak{j}_{21}}^{\mathfrak{b}_{21}}$ is contained inside $\mathcal{B}_{\mathfrak{j}_{12}}^{\mathfrak{b}_{12}}$, we do not need to include the box $\mathcal{B}_{\mathfrak{j}_{21}}^{\mathfrak{b}_{21}}$ in our construction. However, our theorem still works even if we include it. 

The Hilbert space $\mathcal{A}_1$ is a direct sum of three spaces $\mathcal{A}^{11}_1, \mathcal{A}^{12}_1, \mathcal{A}^{22}_1$. The Hilbert space $\mathcal{A}^{11}_1$ is the subspace of $L^2_a(\ball^2)$ spanned by $\{z_2^n, z_1z_2^n, \cdots, z_1^{r-1}z_2^n\}_{n\in \mathbb{N}}$. And the Hilbert space $\mathcal{A}^{12}$ is the finite dimension subspace of $L^2_a(\ball^2)$ spanned by 
\[
\begin{array}{llll}
1&, z_1&, \cdots &, z_1^{p-1},\\
 z_2&, z_1z_2&, \cdots &,z_1^{p-1}z_2, \\
 &\cdots&\cdots&\\
 z_2^{s-1}&,z_1z_2^{s-1}&,\cdots &, z_1^{p-1}z_2^{s-1}.
\end{array}
\]
And the Hilbert space $\mathcal{A}^{22}$ is the subspace of $L^2_a(\ball^2)$ spanned by $\{z_1^n, z_1^n z_2, \cdots, z_1^n z_2^{q-1}\}_{n\in \mathbb{N}}$. 

The Hilbert space $\mathcal{A}_2$ is a direct sum of three spaces,  $\mathcal{A}^{11}_1\cap \mathcal{A}^{12}_1$, $\mathcal{A}^{11}_1\cap \mathcal{A}^{22}_1$,  and $\mathcal{A}^{12}_1\cap \mathcal{A}^{22}_1$. 

The Hilbert space $\mathcal{A}_3$ is the subspace of $L^2_a(\ball^2)$ spanned by 
\[
\begin{array}{lll}
1&,\cdots&, z_1^{r-1},\\
z_2&, \cdots&, z_1^{r-1}z_2,\\
&\cdots& \\
z_2^{q-1}&,\cdots&, z_1^{r-1}z_2^{q-1}. 
\end{array}
\]
\subsection{Ideal $\langle z_1^2, z_3-z_2^2\rangle\subset \complex[z_1, z_2, z_3]$} 
\label{subsec:none-monomial}
We consider the biholomorphic transformation $T:\complex^3\to \complex^3$ by $(z_1, z_2, z_3)\mapsto (u_1, u_2, u_3)=(z_1, z_2, z_3-z_2^2)$. Under the biholomorphic transformation $T$, the ideal $I=\langle z_1^2, z_3-z_2^2\rangle$ is changed to $I'=\langle z_1^2, z_3\rangle$.  $T$ maps the unit ball $\ball^3$ to the domain $\Omega$ defined as follows,
\[
\Omega=\{(u_1, u_2, u_3)| |u_1|^2+|u_2|^2+|u_3+u_2^2|^2<1\}\subset \complex^3. 
\]

Let ${\rm d}V_\Omega$ be the normalized Lebesgue measure on $\Omega$, and $L^2_a(\Omega)$ be the Bergman space on $\Omega$ with respect to the measure ${\rm d}V_\Omega$. Under the biholomorphic transformation $T$, the Bergman space $L^2_a(\ball^3)$ is mapped isometrically to $L^2_a(\Omega)$. And the closure $\overline{I}$ of the ideal $I$ in $L^2_a(\ball^3)$ is mapped to the closure $\overline{I}'$ of the ideal $I'$ in $L^2_a(\Omega)$. We notice that the transformation $T$ maps the polynomial algebra to itself. As $A$-modules,  $L^2_a(\ball^3)$ (and $\overline{I}$) and $L^2_a(\Omega)$ (and $\overline{I}'$) are isomorphic.  

On $L^2_a(\Omega)$, for the ideal $I'=\langle z_1^2, z_3\rangle$, we apply the similar construction as in Section \ref{sec:resolution}. There is only one box $\mathcal{B}_{\mathfrak{j}}^{\mathfrak{b}}$ associated to the ideal $I'$, where $\mathfrak{j}=(1,3)$ and $\mathfrak{b}=(1,0)$, and $\mathcal{B}=\{(n_1, n_2, 0)| n_1\leq 1\}$. We consider the subset $\Omega_{\mathfrak{j}}=\{(z_1, z_2, z_3)\in \Omega| z_1=z_3=0\}$. $\Omega_{\mathfrak{j}}$ can be identified with an open subset of $\complex$ of the form $\{z_2|\ |z_2|^2+|z_2|^4<1\}$. Let $L^2_{a, s}(\Omega_{\mathfrak{j}})$ be the Hilbert space of square integrable holomorphic functions on $\Omega_{\mathfrak{j}}$ with respect to the norm 
\[
(1-|z_2|^2-|z_2|^4)^s {\rm d} V_{\Omega_{\mathfrak{j}}},
\]
where ${\rm d} V_{\Omega_{\mathfrak{j}}}$ is the normalized Lebesgue measure on $\Omega_{\mathfrak{j}}$. 

Define the Hilbert space $\mathcal{A}$ to be the following direct sum of two weighted Bergman spaces, i.e. 
\[
\mathcal{A}=L^2_{a, 2}(\Omega_{\mathfrak{j}})\oplus L^2_{a, 3}(\Omega_{\mathfrak{j}}).
\]

We define the $A=\complex[z_1, z_2, z_3]$-module structure on $\mathcal{A}_{I'}$ by the following formulas. For $(X,Y)\in L^2_{a, 2}(\Omega_{\mathfrak{j}})\oplus L^2_{a, 3}(\Omega_{\mathfrak{j}})$, 
\[
T_{z_1}(X, Y)=(0, X),\ T_{z_2}(X,Y)=(z_2X, z_2Y),\ T_{z_3}(X, Y)=(0,0). 
\]
We notice that monomials $\{z_2^i\}_{i\in \mathbb{N}}$ form an orthogonal basis of both $L^2_{a,2}(\Omega_{\mathfrak{j}})$ and $L^2_{a,3}(\Omega_{\mathfrak{j}})$. A straightforward computation shows that $\mathcal{A}_{I'}$ is an essentially normal Hilbert $\complex[z_1, z_2, z_3]$-module (See also \cite{do-gu-wa:monomial}). 

We define a map $\Psi_{I'}: L^2_a(\Omega)\to \mathcal{A}_{I'}=L^2_{a, 2}(\Omega_{\mathfrak{j}})\oplus L^2_{a, 3}(\Omega_{\mathfrak{j}})$ by 
\[
\Psi(f):=(f|_{z_1=z_3=0}, \frac{\partial f}{\partial z_1}|_{z_1=z_3=0}). 
\]
Similar to the argument in the proof of Theorem \ref{thm:main-intro} in Section \ref{subsec:proof}, we can prove that $\Psi_{I'}$ is a bounded surjective $A$-module morphism and its kernel is exactly $\overline{I}'$.  Hence we have the following short exact sequence of essentially normal $A$-modules, 
\[
0\longrightarrow \overline{I}'\longrightarrow L^2_a(\Omega)\stackrel{\Psi_{I'}}{\longrightarrow} \mathcal{A}_{I'}\longrightarrow 0. 
\]
The biholomorphic transformation $T$ maps $W:=\{(0, z_2, z_2^2)\in \ball^3\}$ isomorphically to $\Omega_{\mathfrak{j}}=\{(0, z_2, 0)\in \Omega\}$. We can replace $\Omega_{\mathfrak{j}}$ by $W$ in the corresponding definitions of $\mathcal{A}_I$ the map $\Psi_I$. Using the transformation $T$, we have a short exact sequence of essentially normal $A$-modules, 
\[
0\longrightarrow \overline{I}\longrightarrow L^2_a(\ball^3)\stackrel{\Psi_I}{\longrightarrow} \mathcal{A}_I\longrightarrow 0. 
\]
As a corollary, we can conclude that the ideal $\overline{I}$ and associated quotient module $Q_I$ are both essentially normal. And we have  the following geometric description for the corresponding Teoplitz extension,
\[
[\mathfrak{T}(Q_I)]=[\mathfrak{T}\big(L^2_{a, 2}(\Omega_{\mathfrak{j}})\oplus L^2_{a, 3}(\Omega_{\mathfrak{j}})\big)]. 
\]

\end{document}